\documentclass[oneside,english]{amsart}

\usepackage{xypic} \usepackage{amssymb} \usepackage{amscd}
\usepackage[dvips]{color}

\numberwithin{equation}{section} %% Comment out for sequentially-numbered
\numberwithin{figure}{section} %% Comment out for sequentially-numbered
  \theoremstyle{plain}
  \newtheorem*{thm*}{Theorem}
  \theoremstyle{plain}
  \newtheorem{thm}{Theorem}[section]
  \theoremstyle{definition}
  \newtheorem{defn}[thm]{Definition}
  \theoremstyle{plain}
  
  \theoremstyle{plain}
  \newtheorem{prop}[thm]{Proposition}
  \theoremstyle{plain}
  \newtheorem{corollary}[thm]{Corollary}
  \theoremstyle{remark}
  
  \theoremstyle{remark}
  \newtheorem*{acknowledgement*}{Acknowledgement}

\def\hol{\mathcal{H}ol} \def\TM{\hat{T}M}

\newcounter{mycounter} 
\let\oldmarginpar\marginpar
\renewcommand\marginpar[1]{\-\oldmarginpar[\raggedleft\footnotesize #1]%
{\raggedright\footnotesize #1}}

\begin{document}

\title[Parameterization-rigidity of the geodesics] {Projective and Finsler metrizability:
  Parameterization-rigidity of the geodesics}

\author[Bucataru]{Ioan Bucataru} \address{Ioan Bucataru, Faculty of
  Mathematics, Al.I.Cuza University \\ B-dul Carol 11, Iasi, 700506,
  Romania} \urladdr{http://www.math.uaic.ro/\textasciitilde{}bucataru/}

\author[Muzsnay]{Zolt\'an Muzsnay} \address{Zolt\'an Muzsnay, Institute
  of Mathematics, University of Debrecen \\ H-4010 Debrecen, Pf. 12,
  Hungary} \urladdr{http://www.math.klte.hu/\textasciitilde{}muzsnay/}

\date{\today}

\begin{abstract} 
  In this work we show that for the geodesic spray $S$ of a Finsler function $F$ the most
    natural projective deformation $\widetilde{S}=S -2 \lambda F\mathbb C$ leads to a
    non-Finsler metrizable spray, for almost every value of $\lambda \in
    \mathbb R$. This result shows how rigid is the
    metrizablility property with respect to certain reparameterizations of
    the geodesics. As a consequence we obtain that the projective class of an arbitrary spray
  contains infinitely many sprays that are not Finsler metrizable.
\end{abstract}


\maketitle

\section{Introduction}
\label{sec:introduction}

A system of second order homogeneous ordinary differential equations
(SODE), whose coefficients functions do not depend explicitly on time,
can be identified with a special vector field, called spray.

The Finsler metrizability problem for a spray $S$ seeks for a Finsler
function whose geodesics coincide with the geodesics of $S$,
\cite{crampin07, krupka85, szilasi02}. In \cite{muzsnay06} a set of
necessary and sufficient conditions for the Finsler metrizability
problem were formulated in terms of the holonomy distribution of a
spray. In this work we will use these conditions to decide wether or not
a spray is Finsler metrizable.

For the projective metrizability problem, one seeks for a Finsler
function whose geodesics coincide with the geodesics of $S$, up to an
orientation preserving reparameterization. The projective metrizability
problem is known as the Finslerian version of Hilbert's fourth problem
\cite{paiva05, crampin08}. In the general case it was Rapcs\'ak
\cite{rapcsak62} who obtained, in local coordinates, necessary and
sufficient conditions for the projective metrizability problem of a
spray.

The two problems can be viewed as particular cases of the inverse
problem of the calculus of variation.  We refer to the review articles
\cite{anderso92, krupkova07, morandi90, sarlet82} for various approaches
of the inverse problem of the calculus of variations. One of this
approaches seeks for the existence of a multiplier matrix that satisfies
four Helmholtz conditions \cite{sarlet82}. In \cite{bd09}, these four
Helmholtz conditions where reformulated in terms of a semi-basic
$1$-form. For the particular case of the Finsler metrizability problem
only three of the Helmholtz conditions are independent \cite{bd09,
  krupka85}, while for the projective metrizability problem, only two
Helmholtz conditions are independent, \cite{bd09}. The formal
integrability of these two Helmholtz conditions was studied in
\cite{bm11} and it lead to some classes of sprays that are projectively
metrizable: isotropic sprays and arbitrary sprays on $2$-dimensional
manifold. Within these classes, we searched for sprays that are not
Finsler metrizable. Of great help for us, at the time, was given by
Yang's example, which was just published online, \cite{yang11}.  Yang
shows that for a flat spray of constant flag curvature its projective
class contains sprays that are not projectively flat and hence cannot be
Finsler metrizable. In this work, using different techniques, we extend
Yang's example, and we show that for an arbitrary spray its projective
class contains sprays that are not Finsler metrizable.

The structure of the paper is as follows. In Section \ref{sec:prelim} we
give a brief introduction of the Fr\"olicher-Nijenhuis theory and the
canonical structures one can define on the tangent bundle of a
manifold. In Section \ref{sec:spraygo} we use the Fr\"olicher-Nijenhuis
theory to introduce the main structures one need to discuss the geometry
of a spray: connection, Jacobi endomorphism, curvature, and covariant
derivative. We pay a special attention to projectively related sprays
and the role of parameterization for the corresponding metrizability problem. In Section \ref{sec:fmp} we discuss the Finsler
metrizability problem and projective metrizability problem for a
spray. For projectively related sprays we provide in Propositions
\ref{prop:sstilde} and \ref{prop:sstilde2} the relations between the
corresponding geometric structures. In Section \ref{sec:classes}, in
Theorem \ref{prop_non_metrizability}, we prove that for an arbitrary
spray $S$, there are infinitely many
values of a scalar $\lambda$ such that the projectively related spray
$\widetilde{S}=S-2\lambda F\mathbb{C}$ is not Finsler metrizable, where
$F$ is a Finsler function and $\mathbb{C}$ the Liouville vector
field. For these values of $\lambda$, we show how to reparameterize
the geodesics of a Finsler function to transform them into
parameterized curves that cannot be the geodesics of any Finsler function.

\section{Preliminaries} \label{sec:prelim}

In this work $M$ is a real and smooth manifold of dimension $n>1$.  We
denote by $C^{\infty}(M)$, the ring of smooth functions on $M$, and by
${\mathfrak X}(M)$, the $C^{\infty}(M)$-module of vector fields on
$M$. Consider $\Lambda(M)=\bigoplus_{k\in {\mathbb N}}\Lambda^k(M)$ the
graded algebra of differential forms on $M$.  We also write
$\Psi(M)=\bigoplus_{k\in {\mathbb N}}\Psi^k(M)$ for the graded algebra
of vector-valued differential forms on $M$.

In this work we will discuss some relations between Finsler and
projective metrizability problems for a homogeneous system of second
order ordinary differential equations using the Fr\"olicher-Nijenhuis
formalism associated to the system. For systematic treatments of the
Fr\"olicher-Nijenhuis theory, we refer to \cite{frolicher56, grifone00,
  KMS93}.  For a vector valued $l$-form $A$ on $M$ consider the inner
product $i_A$, which is a derivation of degree $l-1$ and the Lie
derivation $d_A$, which is a derivation of degree $l$, related by
$$ d_A=i_A\circ d + (-1)^l d\circ i_A.$$ When $l=0$, which means that
$A$ is a vector field, $d_A=\mathcal{L}_A$, the usual Lie
derivative. For two vector valued forms $A\in \Psi^l(M)$ and $B\in
\Psi^s(M)$, the Fr\"olicher-Nijenhuis bracket of $A$ and $B$ is the
unique vector valued $(l+s)$-form $[A,B]$ on $M$ such that
  \begin{eqnarray}d_{[A,B]}=d_A\circ d_B - (-1)^{ls}d_B\circ
    d_A. \label{dab} \end{eqnarray}

  Consider $(TM, \pi, M)$ the tangent bundle of $M$ and
  $(\TM\!:=\!TM\!\setminus\!\{0\}, \pi, M)$ the tangent bundle with the
  zero section removed.  There are some canonical structures one can
  associate to the tangent bundle, such as the vertical distribution,
  Liouville vector field, and tangent structure. We  will
  use the Fr\"olicher-Nijenhuis theory associated to these structures to
  formulate a geometric setting for a a system of SODE, viewed as a
  vector field on the tangent bundle, \cite{bcd11, deleon89, miron94, szilasi03}.

The \emph{vertical distribution} is defined as $V: u\in TM \mapsto
V_u=\{\xi \in T_uTM, d_u\pi(\xi)=0\}$. This distribution is
$n$-dimensional and integrable, since it is tangent to the leaves of the
regular foliation induced by the submersion $\pi$, whose leaves are
tangent spaces to $M$, $\pi^{-1}(p)=T_pM$, for $p\in M$. We denote by
$(x^i)$ local coordinates on the base manifold $M$ and by $(x^i, y^i)$
the induced coordinates on $TM$. It follows that $(y^i)$ are coordinates
in the leaves of the foliation, while $(x^i)$ are transverse coordinates
for the foliation. An important vertical vector field is the
\emph{Liouville vector field}, which locally is given by
$\mathbb{C}=y^i{\partial}/{\partial y^i}.$ The Liouville vector field
will be used to characterize homogeneous objects on $\TM$. For an
integer $s$, we say that a vector valued form $A\in \Psi^l(\TM)$ is
$s$-homogeneous if $\mathcal{L}_{\mathbb{C}}A=(s-1)A$. A form $\omega\in
\Lambda^l(\TM)$ is $s$-homogeneous if
$\mathcal{L}_{\mathbb{C}}\omega=s\omega$.

The \emph{tangent structure} is the $(1,1)$-type tensor field $J$ on
$TM$, which locally is given by $J={\partial}/{\partial y^i}\otimes
dx^i. $ Tensor $J$ satisfies $J^2=0$ and $\operatorname{Ker} J=
\operatorname{Im} J=V$.  The tangent structure $J$ is integrable, which
means that the Fr\"olicher-Nijenhuis bracket $[J,J]$ vanishes. Using
formula \eqref{dab} it follows that $2d_J^2=d_{[J,J]}=0$. Since
$[\mathbb{C}, J]=-J$ it follows that the vector valued $1$-form $J$ is
$0$-homogeneous.

An important class of (vector valued) forms on $\TM$ that are compatible
with the structures presented above is given by semi-basic forms. A form
$\omega$ on $\TM$ is called \emph{semi-basic} if it vanishes whenever
one of its arguments is a vertical vector field. Locally, a semi-basic
$k$-form $\omega$ on $\TM$ can be written as
$$
\omega=\frac{1}{k!}\omega_{i_1...i_k}(x,y)dx^{i_1}\wedge \cdots \wedge
dx^{i_k}.$$ A $1$-form $\omega$ on $\TM$ is semi-basic if and only if
$i_J\omega=\omega \circ J=0$.

A vector valued form $L$ on $\TM$ is called semi-basic if it takes
vertical values and vanishes whenever one of its arguments is a vertical
vector field. In local coordinates, a vector valued semi-basic $l$-form
$L$ on $\TM$ can be written as
$$ L=\frac{1}{l!}L^{j}_{i_1...i_l}(x,y)\frac{\partial}{\partial y^j}\otimes dx^{i_1}\wedge \cdots
\wedge dx^{i_l}.$$ A vector valued $1$-form $L$ on $\TM$ is semi-basic
if and only if $J\circ L=0$ and $i_JL=L \circ J=0$. The tangent
structure $J$ is a vector valued semi-basic $1$-form.

\section{Sprays and related geometric objects} 
\label{sec:spraygo} 

A system of homogeneous second order ordinary differential equations,
whose coefficients functions do not depend explicitly on time, can be
identified with a special vector field on $\TM$ that is called a
spray. In this section we use the Fr\"olicher-Nijenhuis theory to
associate a geometric setting to a spray, \cite{grifone72}. Within this
geometric setting we will discuss in the next sections the Finsler and
projective metrizability problems and some relations between these two
problems.

A vector field $S\in{\mathfrak X}(\TM)$ is called a \emph{spray} if
$JS=\mathbb{C}$ and $[\mathbb{C}, S]=S$.  Locally, a spray can be
expressed as follows
$$ S=y^i\frac{\partial}{\partial x^i} - 2G^i(x,y) \frac{\partial}{\partial
  y^i},$$ for some functions $G^i$ defined on domains of induced
coordinates on $\TM$. The homogeneity condition, $[\mathbb{C}, S]=S$,
for a spray is equivalent with the fact that functions $G^i(x,y)$ are
$2$-homogeneous in the fibre coordinates. In this work we will consider
positive homogeneity only and hence will assume that $G^i(x,\lambda
y)=\lambda^2 G^i(x,y)$ for all $\lambda>0$.

A curve $c:I \to M$ is called \emph{regular} if its tangent lift takes
values in the slashed tangent bundle, $c': I \to \TM$. A regular curve
is called a \emph{geodesic} of spray $S$ if $S\circ c'=c''$. Locally,
$c(t)=(x^i(t))$ is a geodesic of spray $S$ if
\begin{eqnarray} \frac{d^2x^i}{dt^2}+2G^i\left(x,\frac{dx}{dt}\right)=0.
  \label{d2g} \end{eqnarray}
An orientation preserving  reparameterization $t\to \tilde{t}(t)$ of
the system \eqref{d2g} leads to a new spray $\widetilde{S}=S-2P\mathbb
C$, \cite{aim93, shen01}. The scalar function $P\in C^{\infty}(\TM)$ is $1$-homogeneous and
it is related to the new parameter by
\begin{eqnarray}
\frac{d^2 \tilde{t}}{dt^2} = 2P\left(x^i(t),
  \frac{dx^i}{dt}\right)\frac{d\tilde{t}}{dt}, \quad
\frac{d\tilde{t}}{dt}>0. \label{ppar} \end{eqnarray}
\begin{defn}
 Two sprays $S$ and $\widetilde{S}$ are
    \emph{projectively related} if their geodesics coincide up to an
    orientation preserving reparameterization. \end{defn} We will
  refer to the map $S \to \widetilde{S}=S-2P\mathbb{C}$, for $P\in
C^{\infty}(\TM)$ a $1$-homogeneous function, as to the \emph{projective
  deformation} of spray $S$. The aim of this work is to show that
projective deformations, or orientation preserving reparameterizations, are rigid with respect to an important problem
associated to a spray, the Finsler metrizability problem. 

A \emph{nonlinear connection} is defined by an $n$-dimensional
distribution $H: u\in \TM \to H_u\subset T_u(\TM)$ that is supplementary
to the vertical distribution, which means that for all $u\in \TM$, we
have $T_u(\TM)=H_u \oplus V_u$.

Every spray $S$ induces a canonical nonlinear connection through the
corresponding horizontal and vertical projectors, \cite{grifone72}
\begin{eqnarray} h=\frac{1}{2}\left(\operatorname{Id}-[S,J]\right),
  \quad v=\frac{1}{2}\left(\operatorname{Id}+[S,J]\right).
  \label{hv} \end{eqnarray} Equivalently, the canonical nonlinear
connection induced by a spray can be expressed in terms of an
\emph{almost product structure} $\Gamma = -[S,J]=h-v$. With respect to the induced nonlinear
connection, a spray $S$ is horizontal, which means that $S=hS$. Locally,
the two projectors $h$ and $v$ can be expressed as follows
$$h=\frac{\delta}{\delta x^i} \otimes dx^i, \quad
v=\frac{\partial}{\partial y^i}\otimes \delta y^i, \quad
\textrm{where}$$
$$
\frac{\delta}{\delta x^i}=\frac{\partial}{\partial
  x^i}-N^j_i(x,y)\frac{\partial}{\partial y^j}, \quad \delta y^i=dy^i+
N^i_j(x,y)dx^j, \quad N^i_j(x,y)=\frac{\partial G^i}{\partial
  y^j}(x,y).$$ For a spray $S$ consider the vector valued semi-basic
$1$-form
\begin{eqnarray} \Phi= v\circ [S,h] = R^i_j(x,y)
  \frac{\partial}{\partial y^i} \otimes dx^j, \quad R^i_j = 2\frac{\delta
    G^i}{\delta x^j} - S(N^i_j) +
  N^i_kN^k_j, \label{jacobi_end} \end{eqnarray} which will be called the
\emph{Jacobi endomorphism}.

Another important geometric structure induced by a spray $S$ is the
\emph{curvature tensor} $R$. It is the vector valued semi-basic $2$-form
\begin{eqnarray}
  R=\frac{1}{2}[h,h]=\frac{1}{2}R^i_{jk}\frac{\partial}{\partial
    y^i}\otimes dx^j\wedge dx^k , \quad R_{jk}=\frac{\delta N^i_j}{\delta
    x^k} - \frac{\delta N^i_k}{\delta x^j}.\label{curvature}
\end{eqnarray} All geometric objects induced by a spray $S$ inherit the
homogeneity condition. Therefore $[\mathbb{C}, h]=0$, which means that
the nonlinear connection is $1$-homogeneous. Also $[\mathbb{C}, R]=0$,
$[\mathbb{C},\Phi]=\Phi$ and hence the curvature tensor $R$ is
$1$-homogeneous, while the Jacobi endomorphism $\Phi$ is
$2$-homogeneous.

The two vector valued semi-basic $1$ and $2$-forms $\Phi$ and $R$ are
related as follows:
\begin{eqnarray} \Phi=i_SR, \quad
  [J,\Phi]=3R. \label{eq:phir} \end{eqnarray} Locally, the two formulae
\eqref{eq:phir} can be expressed as follows:
\begin{eqnarray} R^i_j=R^i_{kj}y^k, \quad
  R^i_{jk}=\frac{1}{3}\left(\frac{\partial R^i_k}{\partial
      y^j}-\frac{\partial R^i_j}{\partial
      y^k}\right).\label{eq:rr} 
\end{eqnarray} 
For the Jacobi endomorphism $\Phi$ we say that a continuous function
$\kappa\in C^0(\TM)$ is an \emph{eigen function} if there exists a non-zero
horizontal vector field
$X\in \mathfrak{X}(\TM)$ such that $\Phi(X)=\kappa JX$.  The
horizontal vector field $X$ is called an \emph{eigen vector field}. Since the Jacobi
endomorphism is $2$-homogeneous, it follows that its non-zero
eigen functions are $2$-homogeneous functions on $\TM$. See
\cite[p. 58]{grifone00} for more details on the eigen functions and
eigen vector fields for vector valued semi-basic $1$-forms on $TM$. From first
formula \eqref{eq:phir} we obtain that $\Phi(S)=0$ and hence $\kappa=0$
is always an eigen function for the Jacobi endomorphism, the corresponding
eigen vector field is the spray $S$. Therefore, $\operatorname{rank} \Phi\leq
n-1$.

For a spray $S$, consider $\hol_S\subset T(\TM)$ the
\emph{homolnomy distribution} generated by horizontal vector fields and
their successive Lie brackets, \cite{muzsnay06}. If the curvature
  $R$ is non-zero, then $\mathcal{H}ol_S$ contains also vertical
  vector fields. From formula \eqref{curvature} it follows that
$\operatorname{Im} R\subset \hol_S$, while from first formula
\eqref{eq:phir} it follows that $\operatorname{Im} \Phi \subset
\operatorname{Im} R$.

The nonlinear connection induced by a spray $S$ can be characterized
also 
using an \emph{almost complex structure}. It is the $(1,1)$-type tensor
field on $\TM$ given by
\begin{eqnarray*} \mathbb{F}=h\circ [S,h] - J = \frac{\delta}{\delta
    x^i} \otimes \delta y^i - \frac{\partial}{\partial y^i} \otimes
  dx^i. \end{eqnarray*}

For a spray $S$, consider the map $\nabla: \mathfrak{X}(\TM) \to
\mathfrak{X}(\TM)$, given by
\begin{eqnarray} \nabla=h\circ \mathcal{L}_S\circ h + v\circ
  \mathcal{L}_S\circ v =\mathcal{L}_S + h\circ \mathcal{L}_S h + v\circ
  \mathcal{L}_S v \label{nabla}
\end{eqnarray} that will be called the \emph{dynamical covariant
  derivative}. By setting $\nabla f=S(f), \textrm{ for } f\in
C^\infty(\TM)$, using the Leibniz rule, and the requirement that
$\nabla$ commutes with tensor contraction, we extend the action of
$\nabla$ to arbitrary tensor fields and forms on $\TM$, see
\cite[Section 3.2]{bd09}. The action of $\nabla$ on semi-basic forms
coincide with the semi-basic derivation introduced in
\cite[Def. 4.2]{grifone00}. From first formula in \eqref{nabla} it
follows that $\nabla h= 0$ and $\nabla v=0$, which means that $\nabla$
preserves both the horizontal and the vertical distributions. Moreover,
we have $\nabla J=0$, which implies that $\nabla$ has the same action on
horizontal and vertical vector fields. Locally, we can see this from the
following formulae:
\begin{eqnarray*} \nabla \frac{\delta}{\delta
    x^i}=N^j_i\frac{\delta}{\delta x^j}, \quad \nabla
  \frac{\partial}{\partial y^i}=N^j_i\frac{\partial}{\partial
    y^j}. \end{eqnarray*} Using the homogeneity condition $[\mathbb{C},
S]=S$ and formula \eqref{nabla} it follows that $\nabla S=0$ and $\nabla
\mathbb{C}=0$.

Another geometric structure, induced by a spray, and very important for
its geometry, is the \emph{Berwald connection}. It is a linear
connection on $\TM$, and it can be defined as follows $\mathcal{D}:
\mathfrak{X}(\TM) \times \mathfrak{X}(\TM) \to \mathfrak{X}(\TM)$,
\begin{eqnarray} \mathcal{D}_XY=v[hX, vY]\!+\! h[vX, hY] \!+ \!J[vX,
  (\mathbb{F}\!+\!J)Y] \!+\!  (\mathbb{F}\!+\!J)[hX,
  JY]. \label{berwald} \end{eqnarray} Using formula \eqref{berwald}, it
follows that $\mathcal{D}h=0$ and $\mathcal{D}v=0$, which means that the
Berwald connection preserves both the horizontal and vertical
distribution. Moreover, we have $\mathcal{D} J=0$, which implies that
the Berwald connection has the same action on horizontal and vertical
vector fields. Locally, we can see this from the following formulae:
\begin{eqnarray*} \mathcal{D}_{\frac{\delta}{\delta x^i}}
  \frac{\delta}{\delta x^j} = \frac{\partial N^k_i}{\partial y^j}
  \frac{\delta}{\delta x^k}, & \quad \displaystyle
  \mathcal{D}_{\frac{\delta}{\delta x^i}} \frac{\partial}{\partial y^j} =
  \frac{\partial N^k_i}{\partial y^j} \frac{\partial}{\partial y^k}, \\
  \mathcal{D}_{\frac{\partial}{\partial y^i}} \frac{\delta}{\delta x^j} =
  0, & \quad \displaystyle \mathcal{D}_{\frac{\partial}{\partial y^i}}
  \frac{\partial}{\partial y^j} = 0. \end{eqnarray*} Using the fact that
the spray $S$ is horizontal, formulae \eqref{nabla} and \eqref{berwald}
it follows that $\nabla=\mathcal{D}_S$. Therefore, $\mathcal{D}_SS=0$
which means that integral curves of the spray $S$ are geodesics of the
Berwald connection.

\section{Projectively related sprays} \label{sec:fmp}

In this section we discuss the two inverse problems of the calculus of
variations that one can associate to a spray: the Finsler metrizability
problem and the projective metrizability problem. Our aim, in the next
section, will be to search for sprays that are not Finsler metrizable,
within the projective class of a given spray. For this we will need some
formulae that relate the geometric structures of two projectively
related sprays: connections, Jacobi endomorphisms, and curvatures.

\begin{defn} \label{def:finsler} By a \emph{Finsler function} we mean a
  continuous function $F: TM \to \mathbb{R}$ satisfying the following
  conditions:
  \begin{itemize}
  \item[i)] $F$ is smooth on $\TM$;
  \item[ii)] $F$ is positive on $\TM$ and $F(x,0)=0$;
  \item[iii)] $F$ is positively homogeneous of order $1$, which means that
    $F(x,\lambda y)=\lambda F(x,y)$, for all $\lambda>0$ and $(x,y)\in TM$;
  \item[iv)] The \emph{metric tensor} with components
    \begin{eqnarray} g_{ij}(x,y)=\frac{1}{2}\frac{\partial^2 F^2}{\partial
        y^i\partial y^j} \label{gij}
    \end{eqnarray} has rank $n$. \end{itemize}
\end{defn} Conditions ii) and iv) of Definition \ref{def:finsler} imply
that the metric tensor $g_{ij}$ of a Finsler function is positive
definite, \cite{lovas07}.  The regularity condition iv) of Definition
\ref{def:finsler} is equivalent to the fact that the Euler-Poincar\'e
$2$-form of $F^2$, $\omega_{F^2}=dd_JF^2$, is non-degenerate and hence
it is a symplectic structure. Therefore, the equation
\begin{eqnarray} i_Sdd_JF^2=-dF^2 \label{isddj}
\end{eqnarray} uniquely determine a vector field $S$ on $\TM$ that is
called the \emph{geodesic spray} of the Finsler function.
\begin{defn} A spray $S$ is called \emph{Finsler metrizable} if there
  exists a Finsler function $F$ that satisfies the equation \eqref{isddj}.
\end{defn} Necessary and sufficient criteria for the Finsler
metrizability problem for a spray $S$ where formulated in
\cite{muzsnay06} using the holonomy distribution $\hol_S$. We will use
such criteria, in the next section, to construct classes of sprays that
are not Finsler metrizable.

One can reformulate condition iv) of Definition \ref{def:finsler} in
terms of the Hessian of the Finsler function $F$ as follows.  Consider
\begin{eqnarray} h_{ij}(x,y)=F\frac{\partial^2 F}{\partial y^i\partial
    y^j}
  \label{hij}
\end{eqnarray} the \emph{angular metric} of the Finsler function.  Using
the homogeneity of the Finsler function $F$, the metric tensor $g_{ij}$
and the angular tensor $h_{ij}$ are related by
\begin{eqnarray} g_{ij}=h_{ij} + \frac{\partial F}{\partial y^i}
  \frac{\partial F}{\partial y^j}=h_{ij} +
  \frac{1}{F^2}y_iy_j. \label{ghij} \end{eqnarray} In the above formula
\eqref{ghij}, we did use the following calculation, which follows from
the homogeneity of the Finsler function $F$,
\begin{eqnarray} y_i:=g_{ik}y^k=F\frac{\partial F}{\partial
    y^i}=\frac{1}{2}\frac{\partial F^2}{\partial y^i}. \label{y_i}
\end{eqnarray} Throughout this work, we will raise and lower indices
using the metric tensor $g_{ij}$. For the covector field with components
$y_i$, we show now that its horizontal covariant derivative, with
respect to the Berwald connection, vanishes:
\begin{eqnarray} y_{i|j}:=\frac{\delta y_i}{\delta x^j}-\frac{\partial
    N^k_i}{\partial y^j} y_k=0. \label{yij} \end{eqnarray} Indeed we have
$$\left[ \frac{\delta}{\delta x^i}, \frac{\partial}{\partial
    y^j}\right] = \frac{\partial N^k_i}{\partial
  y^j}\frac{\partial}{\partial y^k}. $$ If we apply both sides of this
formula to $F^2$, use formula \eqref{y_i}, and the fact that $d_hF^2=0$
we obtain formula \eqref{yij}.

We consider, the components of the $(1,1)$-type tensor field
\begin{eqnarray} h^i_j:=g^{ik}h_{kj}=
  \delta^i_j-\frac{1}{F^2}y^iy_j.\label{h11} \end{eqnarray} Metric tensor
$g_{ij}$ has rank $n$ if and only if angular tensor $h_{ij}$ has rank
$(n-1)$, see \cite{matsumoto86}. Therefore, the regularity condition of
the Finsler function $F$ is equivalent with the fact that the
Euler-Poincar\'e $2$-form $\omega_{F}=dd_JF$ has rank $2n-2$.
\begin{defn} A spray $S$ is
    \emph{projectively metrizable} if it is projectively related to the
    geodesic spray of a Finsler function.
\end{defn} 
Equivalently, a spray $S$ is projectively metrizable if its geodesics
coincide with the geodesics of a Finsler function, up to an
orientation preserving reparameterization. 

Next proposition, provides the relation between the geometric
structures of two projectively related sprays. We will specialize
these relations in Proposition \ref{prop:sstilde2}, when one of the
spray is the geodesic spray of a Finsler function.
\begin{prop} 
  \label{prop:sstilde} 
  Consider $S$ and $\widetilde{S}$ two projectively related sprays, and
  let $P\in C^{\infty}(\TM)$ be the $ 1$-homogeneous function such that
  $\widetilde{S}=S-2P\mathbb{C}$. The corresponding connections, Jacobi
  endomorphisms, and curvature tensors of the two sprays are related by
  the following formulae: 
  \begin{equation} 
    \label{vvtilde}
    \begin{aligned}
      \widetilde{\Gamma} &=\Gamma -2(PJ + d_JP\otimes \mathbb{C}),
      \\
      \widetilde{h} &=h -PJ - d_JP\otimes \mathbb{C},
      \\
      \widetilde{v} &=v + PJ + d_JP\otimes \mathbb{C},
      \\
      \widetilde{\Phi} &= \Phi +(P^2-S(P))J + (2d_hP-Pd_JP-\nabla
      d_JP)\otimes \mathbb{C},
      \\
      \widetilde{R}&=R + d_Jd_hP\otimes \mathbb{C} + (Pd_JP-d_hP)\wedge
      J.
    \end{aligned}
  \end{equation}
\end{prop}
\begin{proof} 
  Since $\widetilde{S}=S-2P\mathbb{C}$ it follows
  $\widetilde{\Gamma}=-[\widetilde{S}, J]=-[S-2P\mathbb{C},
  J]=-[S,J]+2[P\mathbb{C}, J]$. First formula in \eqref{vvtilde}
  follows, using the homogeneity condition $[\mathbb{C}, J]=-J$. Next
  two formulae are direct consequences of the first one, using the fact
  that $\widetilde{\Gamma}=\widetilde{h}-\widetilde{v}$ and
  $\Gamma=h-v$.

  For the fourth formula, we have to relate
  $\widetilde{\Phi}=\widetilde{v}\circ [\widetilde{S}, \widetilde{h}]$ and
  $\Phi=v\circ [S,h]$. Using the homogeneity conditions of the involved
  geometric objects: $\mathbb{C}(P)=P$, $[\mathbb{C}, S]=S$, $[\mathbb{C},
  h]=0$, and $[\mathbb{C}, J]=-J$ it follows that
  $$ [\widetilde{S}, \widetilde{h}]=[S,h]-PJ + P\Gamma -[S, d_JP]\otimes
  \mathbb{C} + d_JP \otimes \mathbb{C} + 2d_hP \otimes \mathbb{C} - 4Pd_JP
  \otimes\mathbb{C}. $$ Now, if we compose to the left both terms in the
  above formula by $\widetilde{v}$ and use third formula in
  \eqref{vvtilde} we obtain fourth formula in \eqref{vvtilde}. In this
  formula we used the fact that the action of the dynamical covariant
  derivative $\nabla$ on the semi-basic $1$-form $d_JP$ is given by
  $$\nabla(d_JP)=[S,d_JP]-d_vP.$$
  For the last formula in \eqref{vvtilde}, using the relation between
  $\widetilde{\Phi}$ and $\Phi$, we obtain
  \begin{equation}
    \label{3rtilde1}
    \begin{aligned}
      3\widetilde{R} & = [J, \widetilde{\Phi}]
      \\
      & = 3R + d_J(P^2-S(P))\wedge J + [J, (2d_hP-Pd_JP-\nabla
      d_JP)\otimes \mathbb{C}].
    \end{aligned}
  \end{equation} 
  Using the fact that $\omega=2d_hP-Pd_JP-\nabla d_JP$ is a semi-basic
  form, it follows that
  $$ [J, \omega \otimes \mathbb{C}] = d_J\omega \otimes \mathbb{C} -
  \omega \wedge J.$$ In view of these, formula \eqref{3rtilde1} becomes
  \begin{eqnarray} 3\widetilde{R}=3R + d_J\omega \otimes \mathbb{C} +
    \left(-\omega + d_J(P^2-S(P))\right) \wedge
    J.\label{3rtilde2} \end{eqnarray} We compute now $d_J\omega$. We have
  $$d_J\omega = 2d_Jd_hP - d_J(Pd_JP) - d_J\nabla d_JP.$$ Using the
  commutation formula  \begin{eqnarray} \nabla d_J- d_J\nabla = -d_h+
    4i_R \label{nabladj}\end{eqnarray}  
  and the fact that $d_J(Pd_JP)=0$ it follows that
  $d_J\omega=3d_Jd_hP$. Finally, we have
  \begin{eqnarray*} -\omega + d_JP^2 - d_JS(P)= -2d_hP + 3Pd_JP + \nabla
    d_JP - d_J\nabla P= -3d_hP + 3Pd_JP. \end{eqnarray*} If we replace these
  formulae in \eqref{3rtilde2} we obtain that last formula in
  \eqref{vvtilde} is true.
\end{proof} Locally, fourth formula in \eqref{vvtilde} reads as follows
\begin{eqnarray} \widetilde{R}^i_j=R^i_j + (P^2-S(P))\delta^i_j + \left(
    2\frac{\delta P}{\delta x^j} -P\frac{\partial P}{\partial y^j}
    -\nabla\left(\frac{\partial P}{\partial
        y^j}\right)\right)y^i, \label{rrtilde} \end{eqnarray} which is
formula (12.17) in \cite[p.176]{shen01}.

In our work we will be interested in the particular case when the
projective factor is of the form $P=\lambda F$, where $F$ is a Finsler
function and $\lambda$ is a non-zero real number.
\begin{prop} 
  \label{prop:sstilde2} 
  Consider $F$ a Finsler function and let $S$ be its geodesic spray. For
  a non-zero constant $\lambda$, consider the projectively related spray
  $\widetilde{S}=S-2\lambda F \mathbb{C}$. The corresponding
  connections, Jacobi endomorphisms, and curvature tensors of the two
  sprays are related by the following formulae:
  \begin{eqnarray} \nonumber \widetilde{\Gamma} &=&\Gamma -2\lambda(FJ +
    d_JF\otimes \mathbb{C}), \\ \nonumber \widetilde{h} &=&h -\lambda(FJ +
    d_JF\otimes \mathbb{C}), \\
    \label{vvdtilde} \widetilde{v} &=&v + \lambda(FJ + d_JF\otimes
    \mathbb{C}), \\ \nonumber \widetilde{\Phi} &=& \Phi +\lambda^2 (F^2 J -
    F d_JF\otimes \mathbb{C}), \\ \nonumber \widetilde{R}&=&R + \lambda^2 F
    d_JF \wedge J.
  \end{eqnarray}
\end{prop}
\begin{proof} First three formulae in \eqref{vvdtilde} follow from first
  three formulae in \eqref{vvtilde} by replacing $P=\lambda F$.

  Since $S$ is the geodesic spray for the Finsler function $F$, using
  equation \eqref{isddj} it follows that $S(F)=0$ and hence
  $S(P)=0$. Moreover, we have $d_hF=0$ and using the commutation
  formula \eqref{nabladj} it follows
that $\nabla d_JF=0$. Therefore, for $P=\lambda F$, last two formulae in
\eqref{vvtilde} imply the last two formulae in \eqref{vvdtilde}.
\end{proof} Locally, fourth formula in \eqref{vvdtilde}, can be
expressed, using formula \eqref{h11}, as follows:
\begin{eqnarray} \widetilde{R}^i_j = R^i_j + \lambda^2 F^2
  \left(\delta^i_j - \frac{1}{F^2}y^iy_j\right)=R^i_j+\lambda^2 F^2
  h^i_j. \label{rrdtilde} \end{eqnarray} Formula \eqref{rrdtilde}
corresponds to formula \eqref{rrtilde} for the particular case when
the projective factor is  $P=\lambda F$.

\section{Parameterisation-rigidity of the geodesics of a Finsler space}
\label{sec:classes}

In this section we consider $S$ the geodesic spray of a Finsler
function $F$. We show that the most natural projective deformation $S\to \widetilde{S}=S-2\lambda
  F\mathbb C$, $\lambda\in \mathbb R$ leads to non-Finsler metrizable sprays, for almost all values of
  $\lambda$. Consequently, we obtain that the projective class of an
  arbitrary spray contains infinitely many sprays that are not Finsler
  metrizable. This result shows how rigid is the the Finsler metrizablility property
  with respect to certain reparameterization of the geodesics. We provide the corresponding reparameterization of the
  geodesics of a Finsler function $F$ that transforms them into parametrized curves that
  cannot be the geodesics of any Finsler function. 
  \begin{thm}
    \label{prop_non_metrizability}
    Let $S$ be the geodesic spray associated to the Finsler function
    $F$. Then the projective deformation $\widetilde{S}=S -2\lambda F
    \mathbb C$ of $S$ is not Finsler metrizable for almost every value of $\lambda \in
    \mathbb R$. 
  \end{thm}
\begin{proof} 
  Since is $S$ is the geodesic spray of the Finsler function $F$, it
  satisfies the equation \eqref{isddj}. It follows that all Helmholtz
  conditions are satisfied for the spray $S$, \cite{bd09}, and hence the
  Jacobi endomorphism satisfies $d_{\Phi}d_JF^2=0$. In coordinates, this
  Helmholtz condition reads as follows: $g_{ik}R^k_j=g_{jk}R^k_i$. This
  symmetry condition implies that the Jacobi endomorphism is
  diagonalizable. We denote by $r=\operatorname{rank}{\Phi}$, where
  $r\in\{0,...,n-1\}$ and $\kappa_{\alpha}\in C^0(\TM)$, $\alpha\in
  \{1,...,n-1\}$, the eigen functions of $\Phi$ such that the first
  $r$ eigen functions are not zero.

  We fix a point $(x_0, y_0)\in \TM $ and choose $\lambda \in
  \mathbb{R}^*$
  \begin{eqnarray} 
    \lambda^2F^2 +\kappa_{\alpha}\neq 0, \quad \forall
    \alpha \in \{1,...,n-1\}, 
    \label{techniq} 
\end{eqnarray} 
at the point $(x_0, y_0)$.  We remark that almost
  every $\lambda \in \mathbb{R}^*$ can be chosen, since only a finite number
  is not allowed. Due to the
continuity of the Finsler function $F$ and the eigen functions
$\kappa_{\alpha}$, it follows that there is an open subset
$\mathcal{U}\subset \TM$, $(x_0, y_0)\in \mathcal{U}$, such that
condition \eqref{techniq} is satisfied everywhere on $\mathcal{U}$. For
the remaining of the proof, all geometric objects will be considered
restricted to $\mathcal{U}$.

  For the chosen value of $\lambda$, we consider the projectively related
  spray $\widetilde{S}=S-2\lambda F\mathbb{C}$, with the corresponding
  projectors $\widetilde{h}$, $\widetilde{v}$, and Jacobi endomorphism
  $\widetilde{\Phi}$. They are related to the geometric structures induced
  by spray $S$ through formulae \eqref{vvdtilde}. We will prove first that
  $\operatorname{rank}{\widetilde{\Phi}}=n-1$ on $\mathcal{U}$. Consider
  the vector fields
  \begin{eqnarray} h_i=h_i^j\frac{\delta}{\delta x^j}=\frac{\delta}{\delta x^i}-\frac{1}{F^2}y_iS, \quad v_i=Jh_i
  =h^j_i\frac{\partial}{\partial y^j}=\frac{\partial}{\partial
    y^i}-\frac{1}{F^2}y_i\mathbb{C}. \label{hivi}\end{eqnarray}  Since
  $\operatorname{rank}{h^i_j}=n-1$ and $h^i_jy^j=0$ it follows that
  $\mathcal{H}_{n-1}\!:=\!\operatorname{Spann}\{h_1,...,h_n\}$ is an
  $(n-1)$-dimensional horizontal sub-distribution, orthogonal to $S$.
  Similarly, it follows that
  $\mathcal{V}_{n-1}\!:=\!\operatorname{Spann}\{v_1,...,v_n\}$ is an
  $(n-1)$-dimensional vertical sub-distribution, orthogonal to
  $\mathbb{C}$. The above mentioned ortogonality is considered with
  respect to the Sasaki-type metric tensor on $\mathcal{U}$:
  \begin{eqnarray} G=g_{ij} dx^i\otimes dx^j + g_{ij}\delta y^i \otimes
    \delta y^j.\label{sasaki} \end{eqnarray} Therefore, the tangent space to
  $\mathcal{U}$ can be decomposed into four subspaces, orthogonal to each
  other:
  \begin{eqnarray} T\mathcal{U}= \mathcal{H}_{n-1} \oplus
    \operatorname{Spann}\{S\} \oplus \mathcal{V}_{n-1} \oplus
    \operatorname{Spann}\{\mathbb{C}\}. \label{hsvc} \end{eqnarray} From
  formula \eqref{hsvc} it follows that $J \mathcal{H}_{n-1}
  =\mathcal{V}_{n-1} $. 

An important property of the vertical sub-distribution
$\mathcal{V}_{n-1}$ is the following: for any $Y\in \mathcal{V}_{n-1}$
it follows $Y(F)=0$. Indeed, using formulae \eqref{y_i} and \eqref{hivi} it
follows that for any $v_i\in \mathcal{V}_{n-1}$ we have $v_i(F)=0$.

We prove now that the vertical sub-distribution
  $\mathcal{V}_{n-1}$ is integrable. Since $v_i, v_j\in \mathcal{V}_{n-1}
  \subset V$ and the vertical distribution $V$ is integrable, it follows
  that $[v_i, v_j]\in V= \mathcal{V}_{n-1} \oplus
  \operatorname{Spann}\{\mathbb{C}\}$ and hence \begin{eqnarray} [v_i,
    v_j] = A^l_{ij}v_l + B_{ij}\mathbb{C}\label{vivj}, \end{eqnarray} for
  some locally defined functions $A^l_{ij}$ and $B_{ij}$ on
  $\mathcal{U}$. If
  we apply the vector fields in both sides of formula \eqref{vivj} to the
  Finsler function $F$, and use the fact that $v_l(F)=0$ and $\mathbb{C}(F)=F$, we obtain
  $0=B_{ij}F$. This implies that $B_{ij}=0$ and from formula \eqref{vivj}
  it follows that vertical sub-distribution $\mathcal{V}_{n-1}$ is
  integrable.

  For the Jacobi endomorphism $\Phi$, consider $X_{\alpha}\in
  \mathcal{H}_{n-1} $, the eigen vector fields corresponding to the
  eigen functions
  $\kappa_{\alpha}$, $\alpha \in \{1,...,n-1\}$, which means that
  $\Phi(X_{\alpha}) =\kappa_{\alpha}JX_{\alpha}$. Since $X_{\alpha}\in
  \mathcal{H}_{n-1} $ it follows $J X_{\alpha}\in \mathcal{V}_{n-1} $
  and therefore $d_JF(X_{\alpha})=(JX_{\alpha})F=0$.
  Now, using fourth formula \eqref{vvdtilde} it follows that
  $\widetilde{\Phi}(X_{\alpha})=(\lambda^2
  F^2+\kappa_{\alpha})JX_{\alpha}$, for all $\alpha \in
  \{1,...,n-1\}$. With the choice \eqref{techniq} we made for $\lambda$ it follows that
  $\operatorname{Im}\widetilde{\Phi}=\mathcal{V}_{n-1}$ and hence
  $\operatorname{rank}\widetilde{\Phi}= n-1$ on $\mathcal{U}$.

We will prove now that $\widetilde{S}$ is not Finsler metrizable on
$\mathcal{U}$ by showing that its holonomy distribution
$\hol_{\widetilde{S}}$, contains the Liouville vector field
$\mathbb{C}$. We have that $\widetilde{H}=\operatorname{Im}
\widetilde{h}\subset \hol_{\widetilde{S}}$ and
$\mathcal{V}_{n-1}=\operatorname{Im} \widetilde{\Phi} \subset
\hol_{\widetilde{S}}$, which implies
$\widetilde{h}_i=\widetilde{h}(h_i)\in \hol_{\widetilde{S}}$ and
$v_i \in \operatorname{Im}\widetilde{\Phi} \subset
\hol_{\widetilde{S}}$. Therefore we have $[\widetilde{h}_i, v_j]\in
\hol_{\widetilde{S}}$. We will show that one can choose a pair of
indices $(i,j)$ such that the vector field $[\widetilde{h}_i, v_j] $ has
a component along the Liouville vector field.

Since $\widetilde{h}[\widetilde{h}_i, v_j]\in \hol_{\widetilde{S}}$ it
follows that $\widetilde{v}[\widetilde{h}_i, v_j]\in
\hol_{\widetilde{S}}$. From second formula \eqref{vvdtilde} it follows
that $\widetilde{h}_i=h_i-\lambda Fv_i$ and hence we have
\begin{eqnarray} 
  \widetilde{v}[\widetilde{h}_i, v_j]=
  \widetilde{v}[h_i-\lambda Fv_i, v_j] = \widetilde{v}[h_i, v_j] - \lambda
  F [v_i, v_j]. \label{vhivj}\end{eqnarray} For the last equality in the
above formula we did use that $v_j(F)=0$ and hence
$\widetilde{v}[\lambda Fv_i, v_j] = \lambda F \widetilde{v}[v_i, v_j] $.
From third formula \eqref{vvdtilde} it follows that the restrictions of
$v$ and $\widetilde{v}$ to the vertical distribution $V$ coincide. Since
$v_i, v_j \in \mathcal{V}_{n-1}$ and $\mathcal{V}_{n-1}$ is integrable
it follows that $[v_i, v_j]\in \mathcal{V}_{n-1}$. Therefore,
$\widetilde{v}[v_i, v_j] =v[v_i, v_j]=[v_i, v_j] \in
\mathcal{V}_{n-1}\subset \hol_{\widetilde{S}}$, and using formula
\eqref{vhivj} it follows that $ \widetilde{v}[h_i, v_j] \in
\hol_{\widetilde{S}}$.

  Using third formula \eqref{vvdtilde}, we obtain
  \begin{eqnarray} \widetilde{v}[h_i, v_j] = v[h_i, v_j] + \lambda F
    J[h_i, v_j] + \lambda (J[h_i, v_j])(F) \mathbb{C} \in
    \hol_{\widetilde{S}}. \label{vhivj1} \end{eqnarray} Using formula
  \eqref{berwald} it follows that the first two vector fields in the right
  hand side of the above formula can be expressed in terms of the Berwald
  connection as follows
  \begin{eqnarray} v[h_i, v_j]=\mathcal{D}_{h_i}v_j=
    h^l_ih^k_{j|l}\frac{\partial}{\partial y^k}, \quad J[h_i,
    v_j]=-\mathcal{D}_{v_j}v_i = v_i(h^k_j) \frac{\partial}{\partial
      y^k}. \label{vhivj2} \end{eqnarray} In formula \eqref{vhivj2},
  $h^k_{j|l}$ represents the horizontal covariant derivative of the
  $(1,1)$-type tensor field $h^k_j$ with respect to the berwald
  connection:
  $$ h^k_{j|l}=\frac{\delta h^k_j}{\delta x^l} + h^i_j\frac{\partial
    N^k_i}{\partial y^l} - h^k_i\frac{\partial N^i_j}{\partial y^l}. $$ We
  show that $\mathcal{D}_{h_i}v_j\in \mathcal{V}_{n-1}\subset
  \hol_{\widetilde{S}}$. Since the Berwald connection preserves the
  vertical distribution it follows that $\mathcal{D}_{h_i}v_j\in V$ and in
  view of the orthogonal decomposition $V=\mathcal{V}_{n-1} \oplus
  \operatorname{Spann}\{\mathbb{C}\}$, it remains to show that
  $G(\mathcal{D}_{h_i}v_j, \mathbb{C})=0$. From first formula
  \eqref{vhivj2} and formula \eqref{sasaki} we have
  \begin{eqnarray} G\left(\mathcal{D}_{h_i}v_j, \mathbb{C}\right) =
    G\left(h^l_ih^k_{j|l}\frac{\partial}{\partial y^k},
      y^s\frac{\partial}{\partial y^s}\right)=h^l_ih^k_{j|l}g_{ks}y^s =
    h^l_ih^k_{j|l}y_k=0, \label{ghh} \end{eqnarray} For the last equality in
  formula \eqref{ghh} we did use that $h^k_jy_k=0$ and hence its
  horizontal covariant derivative with respect to the Berwald connection
  is zero as well: $0=(h^k_jy_k)_{|l}=h^k_{j|l}y_k +
  h^k_jy_{k|l}=h^k_{j|l}y_k$, since due to formula \eqref{yij} we have
  $y_{k|l}=0$.

  Let us evaluate now, the vector field $J[h_i, v_j]$ using the second
  formula \eqref{vhivj2}. Using formula \eqref{h11}, we have
  \begin{eqnarray} J[h_i, v_j]=-\mathcal{D}_{v_j}v_i= \frac{1}{F^2}y_iv_j
    + \frac{1}{F^2} h_{ij} \mathbb{C}.\label{jhivj} \end{eqnarray} Using
  formula \eqref{jhivj} and the fact that $v[h_i,v_j]\in
  \hol_{\widetilde{S}}$, it follows that the last two terms in formula
  \eqref{vhivj} can be written as follows
  \begin{eqnarray} \hol_{\widetilde{S}} \ni \lambda F J[h_i, v_j] +
    \lambda (J[h_i, v_j])(F) \mathbb{C} = \frac{\lambda}{F} y_iv_j +
    \frac{2\lambda}{F}h_{ij} \mathbb{C},\label{vhivj3} \end{eqnarray} which
  implies that $h_{ij} \mathbb{C} \in \hol_{\widetilde{S}} $ for all pairs
  of indices $(i,j)$. Since $\operatorname{rank}(h_{ij})=n-1$ and $n>1$ it
  follows that there is at least one pair $(i,j)$ such that $h_{ij}\neq
  0$. Therefore $\mathbb{C}\in \hol_{\widetilde{S}}$ and
  according to Theorem 2 of \cite{muzsnay06} this proves that the
  restriction of $\widetilde{S}$ to $\mathcal{U}$ is not Finsler
  metrizable and hence the spray $\widetilde{S}$ is not Finsler
  metrizable.
\end{proof}
A direct consequence of the Theorem \ref{prop_non_metrizability} is given by the following corollary.
\begin{corollary}
    \label{maint} 
    For any spray its projective class contains infinitely many sprays that
    are not Finsler metrizable. 
  \end{corollary}

In the case when the geodesic spray $S$ of a Finsler function $F$ has
constant flag curvature $\kappa$ \cite{bcs00}, it follows that all eigen functions of
the Jacobi endomorphism are $\kappa_{\alpha}=\kappa F^2$, $\alpha \in
\{1,..., n-1\}$. In this case, the condition \eqref{techniq} becomes
\begin{eqnarray} \kappa+\lambda^2\neq 0. \label{kl} \end{eqnarray} It
follows that one can choose $\lambda \in \mathbb{R}^*$ such that
condition \eqref{kl} and hence condition \eqref{techniq} is satisfied
everywhere on $\TM$.

The particular case when the geodesic spray $S$ of a Finsler function
$F$ is flat and has constant flag curvature $\kappa$ was studied by Yang
in \cite{yang11} using different techniques but the same condition
\eqref{kl}. Yang's example has been used also in \cite{bm11} to provide
examples of sprays that are projectively metrizable and not Finsler
metrizable.

We will show now how to reparameterize the geodesics of a Finsler
function $F$ such that the new parameterized curves are not the
geodesics for any Finsler function. Consider $F$ a Finsler function with geodesic equations given by the
system \eqref{d2g}, where $t$ is the arc length of the Finsler function
$F$. Consider $\lambda\in \mathbb R^*$ satisfying the
condition \eqref{techniq}. According to Theorem \ref{prop_non_metrizability} we have to
search for a new parameterization $\tilde{t}$ that satisfies the
equation \eqref{ppar}. It follows that the reparameterization $\tilde{t}=c_1t+2\lambda c_2e^{2\lambda t}$ of the system \eqref{d2g}
leads to a system of second order differential equations that is not Finsler metrizable, 
for $c_1, c_2$ real constants such that $c_1>0$ and $\lambda c_2>0$.

\begin{acknowledgement*} The work of I.B. has been supported by the
  Romanian National Council of Scientific Research (CNCS) Grant PN II ID
  PCE 0291. The work of Z.M. has been supported by the Hungarian
  Scientific Research Fund (OTKA) Grant K67617.
\end{acknowledgement*}

\end{document}